\documentclass{amsart}

\newtheorem{theorem}{Theorem}[section]
\newtheorem{lemma}[theorem]{Lemma}

\theoremstyle{definition}

\newtheorem{example}[theorem]{Example}
\newtheorem{proposition}[theorem]{Proposition}
\newtheorem{corollary}[theorem]{Corollary}

\theoremstyle{remark}
\newtheorem{remark}[theorem]{Remark}

\numberwithin{equation}{section}

\begin{document}

\title{Three circles theorems and Liouville type theorems}

\author{Run-Qiang Jian}
\address{School of Computer Science, Dongguan University
of Technology, 1, Daxue Road, Songshan Lake, 523808, Dongguan, P.
R. China}
\email{jianrq@dgut.edu.cn}

\author{Zhu-Hong Zhang}
\address{School of Mathematical Sciences, South China Normal University, Guangzhou, 510631, China}
\email{juhoncheung@sina.com}





\begin{abstract}
We establish three circles theorems for subharmonic functions on Riemannian manifolds with nonnegative Ricci curvature, as well as on gradient shrinking Ricci solitons with scalar curvature bounded from below by $\frac{n-2}{2}$. We also establish a three circiles theorem for holomorphic functions on gradient shrinking K\"{a}hler-Ricci solitons with some curvature conditions. As applications, we prove some Liouville type theorems.
\end{abstract}

\maketitle

\section{Introduction}

In complex analysis, the Hadamard three circles theorem states that the logarithm of the maximal modulus of a holomorphic function on $\mathbb{C}$ is convex with respect to the logarithm of the radius of a circle. It can be generalized to subharmonic functions on $\mathbb{R}^n$ (see e.g., Theorem 28 and 30 in Chapter 2 of \cite{PW}). During the last decades, classical three circles theorems have been generalized to solutions of partial differential equations (see, e.g., \cite{S}, \cite{CT}, \cite{CM}, and \cite{X}). Recently, in a seminal paper \cite{L}, Liu established a three circles theorem on K\"{a}hler manifolds with nonnegative holomorphic sectional curvature. Later on, Liu's theorem was generalized to shrinking gradient Ricci solitons with constant scalar curvature (\cite{O}), pseudohermitian manifolds (\cite{CHL}),  and almost Hermitian manifolds (\cite{YZ}).

Classical three circles theorems can be formerly used to prove Liouville type theorems. One important application of Liu's three circles theorem and its generalizations is to derive sharp dimension estimates (see \cite{L}, \cite{CHL}, and \cite{YZ}).

The main purpose of the present paper is to generalize three circles theorems to Riemannian manifolds. We first consider the generalization of three circles theorems for subharmonic functions on $\mathbb{R}^n$. An natural and appropriate framework is Riemannian manifolds with nonnegative Ricci curvature. Let $M$ be an $n$-dimensional Riemannian manifold. Given a subharmonic function $u$ on $M$, we denote by $\mathrm{M}_u(r)$ the maximal value of $u$ over a geodesic ball centered at a fixed point $p\in M$ with radius $r$. We show that
\begin{theorem}Suppose that $M$ is a Riemannian manifold with nonnegative Ricci curvature and $u$ is a subharmonic function on $M$. For any $0<r_1<r_2<r_3$,
 if $n=2$, we have \begin{equation*}\mathrm{M}_u(r_2)\leq\frac{\mathrm{M}_u(r_1)(\log r_3-\log r_2)+\mathrm{M}_u(r_3)(\log r_2-\log r_1)}{\log r_3-\log r_1},\end{equation*} and if $n>3$, we have \begin{equation*}\mathrm{M}_u(r_2)\leq\frac{\mathrm{M}_u(r_1)(r_3^{2-n}-r_2^{2-n})+\mathrm{M}_u(r_3)(r_2^{2-n}-r_1^{2-n})}{r_3^{2-n}-r_1^{2-n}}.\end{equation*}
\end{theorem}
As a consequence, we have the following Liouville type theorem.

\begin{corollary}Any bounded subharmonic function on a surface with nonnegative Gauss curvature must be a constant. \end{corollary}

The most essential ingredient in the proof of three circles theorems in $\mathbb{R}^n$ is the fundamental solutions for Laplace equations. Contrary to the classical case, there is a lack of such functions on general Riemannian manifolds. In order to overcome this difficulty, we combine Liu's trick in \cite{L} and the Laplacian comparison theorem.

On the other hand, three circles theorems involve the distant function $d(\cdot, p)$ with respect to some reference point $p$. But usually $d(\cdot, p)$ does not perform well enough. A suitable replacement of $d(\cdot, p)$ is the potential functions of gradient Ricci solitons. Recall that a Riemannian manifold $(M, g)$, equipped with a specific smooth function $f$, is called gradient Ricci soliton, if the equation $$\mathrm{Ric}+\mathrm{Hess}(f)=\lambda g$$ holds for some constant $\lambda$. Here $\mathrm{Ric}$ and $\mathrm{Hess}(f)$ denote the Ricci curvature of $M$ and the Hessian of $f$ respectively. A soliton is called shrinking, steady, and expanding, accordingly, if $\lambda>0$, $\lambda=0$, and $\lambda<0$. By scaling the metric $g$, one customarily assumes $\lambda\in \{-\frac{1}{2}, 0, \frac{1}{2}\}$. Gradient shrinking Ricci solitons, which can be viewed as a generalization of the Einstein matric $\mathrm{Ric}=\lambda g$,  play an important role in Ricci flow, as they often arise as Type-I singularity models to the Ricci flow (\cite{EMT}). We refer the readers to \cite{C} for a quick overview and more information. Thus one is interested in possibly classifying them or understanding their geometry. Recently, there has been a lot of research.

On a gradient shrinking Ricci soliton,  $\rho=2\sqrt{f}$ behaves like $d(\cdot, p)$. However, $\rho$ is smooth on $M$. Moreover, it satisfies certain equations related to Ricci curvature of $M$, which enable us to handle the problem more easily. We concern three circles theorems on gradient shrinking Ricci solitons. We prove the following result:
\begin{theorem}For any subharmonic function $u$ on a gradient shrinking soliton $(M,g,f)$, there exists $r_0>0$ such that for any $r_0\leq r_1<r_2<r_3$, we have $$\mathrm{M}_u(r_2)\leq\frac{\mathrm{M}_u(r_1)\big(\Phi(r_3)-\Phi(r_2)\big)+\mathrm{M}_u(r_3)\big(\Phi(r_2)-\Phi(r_1)\big)}{\Phi(r_3)-\Phi(r_1)}.$$Here $\mathrm{M}_u(r)$ is the maximal value of $u$ over the set $\{q\in M|\rho(q)\leq r\}$, and $\Phi(r)=\log r$ if $n=2$ and $=\frac{r^{2-n}}{2-n}$ if $n>2$.\end{theorem}

Furthermore, we have

\begin{theorem}For any subharmonic function $u$ on a gradient shrinking soliton $(M,g,f)$ with scalar curvature $S\geq \frac{n-2}{2}$  and any $5n<r_1<r_2<r_3$, we have \begin{equation*}\mathrm{M}_u(r_2)\leq\frac{\mathrm{M}_u(r_1)\big(\log \log r_3- \log \log r_2 \big)+\mathrm{M}_u(r_3)\big(\log \log r_2 - \log \log r_1 \big)}{ \log \log  r_3 -\log \log  r_1}.\end{equation*} \end{theorem}

A Liouville type theorem under the above conditions follows immediately (see Corollary \ref{Liouville 1} below).

In \cite{MW}, Munteanu and Wang establish Liouville-type results on complete K\"{a}hler manifolds admitting a real holomorphic gradient vector field. In particular, they show that a bounded holomorphic function on a shrinking gradient K\"{a}hler-Ricci soliton must be constant. We characterize holomorphic functions on shrinking gradient K\"{a}hler-Ricci solitons from the viewpoint of three circles theorem. Let $(M,\omega,f)$ be a gradient shrinking K\"{a}hler-Ricci soliton. For any holomorphic function $h$ on $M$, we denote by $\mathrm{M}_h(r)$ the maximal modulus of $h$ over the set $\{q\in M|\rho(q)\leq r\}$.

\begin{theorem}Let $(M,\omega,f)$ be a gradient shrinking K\"{a}hler-Ricci soliton. Suppose that $\nabla f\neq 0$ and $\mathrm{Ric}(e_\rho^{1,0}, \overline{e_\rho^{1,0}})\geq \frac{S}{f}$ outside a compact subset of $M$. Then there exists $r_0>0$ such that for any $r_0\leq r_1<r_2<r_3$, \begin{equation*}\log\mathrm{M}_h(r_2)\leq\frac{\log\mathrm{M}_h(r_1)(\log r_3-\log r_2)+\log\mathrm{M}_h(r_3)(\log r_2-\log r_1)}{\log r_3-\log r_1},\end{equation*}where $e_\rho^{1,0}$ is the image of $\frac{\nabla \rho}{|\nabla \rho|}$ under the natural correspondence $TM\rightarrow T^{1,0}M$. \end{theorem}

We establish a Liouville type theorem for holomorphic functions on such manifolds (see Corollary \ref{Liouville 2} below).

This paper is organized as follows. In Section 3, we establish three circles theorems for subharmonic functions on Riemannian manifolds with nonnegative Ricci curvature, as well as a Liouville type theorem for surfaces with nonnegative Gauss curvature. In Section 4, three circles theorems and Liouville type theorems are proved for subharmonic functions on gradient shrinking Ricci solitons with some suitable conditions on scalar curvature. In the final section, we consider holomorphic functions on gradient shrinking K\"{a}hler-Ricci solitons. Under some conditions of curvature, three circles theorems and Liouville type theorems are established.

\section{Nonnagtive Ricci condition}
Let $(M^n,g)$ be a connected complete $n$-dimensional Riemannian manifold with $n\geq 2$. We denote by $\Delta$ the Laplace-Beltrami operator on $M$, i.e., $\Delta=\mathrm{div}\circ \nabla$. A $C^2$ function $u$ on $M$ is called subharmonic if $\Delta u\geq 0$. For subharmonic functions, we have the following strong maximum principle (see, e.g., Theorem 7.1.7 in \cite{P}).
\begin{proposition}\label{strong maximum principle}If $u$ is a subharmonic function on $M$, then $u$ is constant in a neighborhood of every local maximum. In particular, if $u$ has a global maximum, then $u$ is constant. \end{proposition}

Let $d$ be the distance function induced from $g$. Fix a point $p\in M$. In this section, we denote by $\rho(\cdot)=d(p,\cdot)$ the distance function with respect to $p$. For any $r>0$, we denote $$B_p(r)=\{q\in M|\rho(q)<r\}.$$ For a given subharmonic function $u$ on $M$, we denote $$\mathrm{M}_u(r)=\max_{\overline{B}_p(r)}u,$$where $\overline{B}_p(r)=\{q\in M|\rho(q)\leq r\}$.

\begin{lemma}\label{lemma 1}Suppose that $u$ is a nonconstant subharmonic function on $M$. Then, as a function of $r$, $\mathrm{M}_u(r)$ is strictly increasing.\end{lemma}
\begin{proof}If not, then there exist $0<r_1<r_2$ such that $\mathrm{M}_u(r_1)\geq\mathrm{M}_u(r_2)$. In other words, $u$ attains a global maximum inside the ball $B_p(r_2)$, which contradicts to the above strong maximum principle.\end{proof}

For $n\geq 2$, we denote  \[\Phi(r)=\Phi_n(r)=\begin{cases}\log r,&\text{if }n=2;\\[3pt]\frac{r^{2-n}}{2-n},&\text{if }n>2.\end{cases}\]Then $$\Phi'(r)=\frac{1}{r^{n-1}}$$and $$\Phi''(r)=\frac{1-n}{r^{n}}.$$Obviously, $\Phi$ is an increasing function on $(0,\infty)$. Given $r>0$ and $0<\varepsilon<r$, we consider the function $\Phi(\rho-\varepsilon)$ on $\{q\in M|\rho(q)\geq r\}$. We have

\begin{lemma}\label{lemma 2}If the Ricci curvature of $M$ is nonnegative, then $$\Delta \Phi(\rho-\varepsilon)\leq \frac{(1-n)\varepsilon}{\rho(\rho-\varepsilon)^n}$$ on $M\setminus (\mathrm{Cut}(p)\cup\{p\})$. Here $\mathrm{Cut}(p)$ is the cut locus of $p$ on $M$.\end{lemma}
\begin{proof}On $M\setminus (\mathrm{Cut}(p)\cup\{p\})$, we have \begin{align*}
\Delta \Phi(\rho-\varepsilon)&=\nabla^i\nabla_i \Phi(\rho-\varepsilon)\\[3pt]
&=\nabla^i\big( \Phi'(\rho-\varepsilon)\nabla_i\rho\big)\\[3pt]
&=\Phi''(\rho-\varepsilon)\nabla^i\rho\nabla_i\rho+ \Phi'(\rho-\varepsilon)\nabla^i\nabla_i\rho\\[3pt]
&=\Phi''(\rho-\varepsilon)|\nabla \rho|^2+\Phi'(\rho-\varepsilon)\Delta \rho\\[3pt]
&=\frac{1-n}{(\rho-\varepsilon)^n}+\frac{1}{(\rho-\varepsilon)^{n-1}}\Delta \rho.
\end{align*}where the last equality follows from the well-known fact that $\|\nabla \rho\|=1$ on $M\setminus (\mathrm{Cut}(p)\cup\{p\})$.

Since $\mathrm{Ric}(M)\geq 0$, by the Laplace comparison theorem, we have $$\Delta \rho\leq \frac{n-1}{\rho}$$ on $M\setminus (\mathrm{Cut}(p)\cup\{p\})$. Therefore$$
\Delta h_\varepsilon\leq \frac{1-n}{(\rho-\varepsilon)^n}+\frac{1}{(\rho-\varepsilon)^{n-1}}\frac{(n-1)}{\rho}=\frac{(1-n)\varepsilon}{\rho(\rho-\varepsilon)^n}$$
on $M\setminus (\mathrm{Cut}(p)\cup\{p\})$.\end{proof}

\begin{theorem}\label{three-circles1}Suppose that the Ricci curvature of $M$ is nonnegative and $u$ is a subharmonic function on $M$. Then for any $0<r_1<r_2<r_3$, we have $$\mathrm{M}_u(r_2)\leq\frac{\mathrm{M}_u(r_1)\big(\Phi(r_3)-\Phi(r_2)\big)+\mathrm{M}_u(r_3)\big(\Phi(r_2)-\Phi(r_1)\big)}{\Phi(r_3)-\Phi(r_1)}.$$\end{theorem}
\begin{proof}For any $0<r_1<r_3$, we denote $$A(r_1,r_3)=\{q\in M|r_1\leq\rho (q)\leq r_3\}.$$We define two functions $G$ and $F$ on $A(r_1,r_3)$ respectively by $$G=\big(\Phi(r_3)-\Phi(r_1)\big)u,$$and $$F=\big(\mathrm{M}_u(r_3)-\mathrm{M}_u(r_1)\big)\Phi(\rho)+\mathrm{M}_u(r_1)\Phi(r_3)-\mathrm{M}_u(r_3)\Phi(r_1).$$We are going to show that $G\leq F$ on $A(r_1,r_3)$, from which our conclusions follow.

For any $0<\varepsilon<r_1$, we define $$F_\varepsilon=a_\varepsilon \Phi(\rho-\varepsilon)+b_\varepsilon=a_\varepsilon h_\varepsilon+b_\varepsilon,$$where $a_\varepsilon$ and $b_\varepsilon$ are real numbers such that $F_\varepsilon(q)=\big(\Phi(r_2)-\Phi(r_1)\big)\mathrm{M}_u(r_1)$ for $q\in \partial B_p(r_1)$ and $F_\varepsilon(q)=\big(\Phi(r_3)-\Phi(r_1)\big)\mathrm{M}_u(r_3)$ for $q\in \partial B_p(r_3)$. It is easy to see that $$a_\varepsilon=\frac{\big(\Phi(r_3)-\Phi(r_1))(\mathrm{M}_u(r_1)-\mathrm{M}_u(r_3))}{\Phi(r_1-\varepsilon)-\Phi(r_3-\varepsilon)},$$ and $$b_\varepsilon=\big(\Phi(r_3)-\Phi(r_1)\big)\frac{\mathrm{M}_u(r_1)\Phi(r_3-\varepsilon)-\mathrm{M}_u(r_3)\Phi(r_1-\varepsilon)}{\Phi(r_3-\varepsilon)-\Phi(r_1-\varepsilon)}.$$Since $\Phi(r)$ and $\mathrm{M}_u(r)$ are increasing in $r$, $a_\varepsilon>0$. Obviously, $F_\varepsilon\rightarrow F$ as $\varepsilon\rightarrow 0^+$.

Assume that $(G-F_\varepsilon)(q)>0$ for some $q\in A(r_1,r_3)$. Since $G-F_\epsilon$ is continuous on $A(r_1,r_3)$, it attains its maximum at some $\tilde{q}_\varepsilon\in A(r_1,r_3)$. Moreover $(G-F_\varepsilon)(\tilde{q}_\varepsilon)>0$. Note that for $q\in \partial B_p(r_1)$, $$
(G-F_\varepsilon)(q)=\big(\Phi(r_3)-\Phi(r_1)\big)\big(u(q)-\mathrm{M}_u(r_1)\big)\leq 0.$$Similarly, for $q\in \partial B_p(r_3)$, $$
(G-F_\varepsilon)(q)=\big(\Phi(r_3)-\Phi(r_1)\big)\big(u(q)-\mathrm{M}_u(r_3)\big)\leq 0.$$So $G-F_\varepsilon$ is always non-positive on the boundary of $A(r_1,r_3)$, and hence $\tilde{q}_\varepsilon$ is an interior point of $A(r_1,r_3)$.

We first assume that $\tilde{q}_\epsilon\notin \mathrm{Cut}(p)$. It follows that $$\Delta(G-F_\varepsilon)(\tilde{q}_\varepsilon)\leq 0.$$On the other hand, Since $u$ is subharmonic, $\Delta G\geq 0$. So, by Lemma \ref{lemma 1} and \ref{lemma 2}, we have \begin{align*}
\Delta(G-F_\epsilon)(\tilde{q}_\varepsilon)&=\Delta G(\tilde{q}_\varepsilon)-a_\varepsilon\Delta h_\varepsilon(\tilde{q}_\varepsilon)\\[3pt]
&\geq a_\epsilon\frac{(1-n)\varepsilon}{\rho(\tilde{q}_\varepsilon)\big(\rho(\tilde{q}_\varepsilon)-\varepsilon\big)^n}\\[3pt]
&>0,
\end{align*}a contradiction. Hence $G-F_\varepsilon\leq 0$ on $A(r_1,r_3)$. By letting $\varepsilon\rightarrow 0^+$, we have $G-F\leq 0$ on $A(r_1,r_3)$, as desired.

If $\tilde{q}_\epsilon\in \mathrm{Cut}(p)$, we choose a number $0<\varepsilon_1<\varepsilon$ and a point $p_1$ on a fixed minimal geodesic connecting $p$ and $\tilde{q}_\varepsilon$ such that $d(p,p_1)=\varepsilon_1$. Define $$F_{\varepsilon,\varepsilon_1}=a_\varepsilon \Phi(\hat{\rho}+\varepsilon_1-\varepsilon)+b_\varepsilon,$$where $\hat{\rho}(\cdot)=d(p_1,\cdot)$. Then for any $q\in M\setminus \{p\}$, \begin{align*}
F_\varepsilon(q)&=a_\varepsilon \Phi(d(p,q)-\epsilon)+b_\varepsilon\\[3pt]
&\leq a_\varepsilon \Phi(d(p,p_1)+d(p_1,q)-\epsilon)+b_\varepsilon\\[3pt]
&=F_{\varepsilon,\varepsilon_1}(q).
\end{align*}Furthermore, since $F_{\varepsilon,\varepsilon_1}(\tilde{q}_\epsilon)=F_\varepsilon(\tilde{q}_\varepsilon)$, $\tilde{q}_\varepsilon$ is the maximum point of $G-F_{\varepsilon,\varepsilon_1}$ on $A(r_1,r_2)$ as well. Thus $$\Delta(G-F_{\varepsilon,\varepsilon_1})(\tilde{q}_\varepsilon)\leq 0.$$On the other hand, \begin{align*}
\Delta(G-F_{\varepsilon,\varepsilon_1})(\tilde{q}_\varepsilon)&=\Delta G(\tilde{q}_\varepsilon)-a_\epsilon\Delta h_{\varepsilon-\varepsilon_1}(\tilde{q}_\varepsilon)\\[3pt]
&\geq a_\varepsilon\frac{(1-n)(\varepsilon-\varepsilon_1)}{\rho(\tilde{q}_\varepsilon)\big(\rho(\tilde{q}_\varepsilon)-(\varepsilon-\varepsilon_1)\big)^n}\\[3pt]
&>0,
\end{align*}a contradiction. Again, we have $G-F_\varepsilon\leq 0$ on $A(r_1,r_2)$, and hence $G-F\leq 0$ on $A(r_1,r_2)$.\end{proof}

\begin{remark}If $M=\mathbb{R}^n$ is the usual Euclidean space, then the above theorem coincides with the classical one (see, e.g., Theorem 30 in Chapter 2 in \cite{PW}).\end{remark}

\begin{corollary}If $M$ is a surface with nonnegative Gauss curvature, then any bounded subharmonic function on $M$ must be a constant. \end{corollary}
\begin{proof}Assume that $u$ is a bounded subharmonic function on $M$. For $0<r_1<r_2<r_3$, note that $$\lim_{r_3\rightarrow \infty}\frac{\log r_3-\log r_2}{\log r_3-\log r_1}=1,$$ and$$\lim_{r_3\rightarrow \infty}\mathrm{M}_u(r_3)\frac{\log r_2-\log r_1}{\log r_3-\log r_1}=0.$$By letting $r_3\rightarrow \infty$ in the above theorem for case $n=2$, we have immediately that $$\mathrm{M}_u(r_2)\leq \mathrm{M}_u(r_1).$$It means that $u$ attains its maximum at some interior point of $\overline{B}_p(r_2)$. So $u$ is a constant.\end{proof}

\section{Gradient shrinking Ricci solitons}

In this section, we always assume that $(M,g,f)$ is an $n$-dimensional complete gradient shrinking Ricci soliton with a proper potential $f$. Thus \begin{equation}\mathrm{Ric}+\mathrm{Hess} (f)=\frac{g}{2}.\end{equation}
We recall some basic identities of complete gradient shrinking Ricci solitons.

\begin{lemma}[\cite{H}]
Let $S$ be the scalar curvature of $M$. Then we have
$$S+\Delta f =\frac{n}{2},$$ and
$$S+|\nabla f|^2=f.$$
\end{lemma}

\begin{lemma}[\cite{C}]The scalar curvature of $M$ is always nonnegative, i.e., $S\ge 0$.
\end{lemma}

\begin{lemma}[\cite{CZ}]Fix $p\in M$. Then there are positive constant $c$ and $r_0$ such that
 $$\frac{1}{4}\Big(d(x, p)-c\Big)^2\leq f(x)\leq \frac{1}{4}\Big(d(x, p)+c\Big)^2$$ whenever $d(x,p)\geq r_0$.
\end{lemma}

Indeed, both constants $c$ and $r_0$ above can be chosen to depend only on the dimension $n$.

In this section, we denote $\rho=2\sqrt{f}$. For any $r>0$, denote $D_r=\{q\in M|\rho(q)\leq r\}$. Given a subharmonic function $u$ on $M$, we set $$\mathrm{M}_u(r)=\max_{D_r}u.$$ If $u$ is not a constant, then for any $0<r_1<r_2$, $\mathrm{M}_u(r_1)<\mathrm{M}_u(r_2)$. As before, we set  \[\Phi(r)=\Phi_n(r)=\begin{cases}\log r,&\text{if }n=2;\\[3pt]\frac{r^{2-n}}{2-n},&\text{if }n>2.\end{cases}\]

\begin{lemma}\label{lemma 3}There exists $r_0>0$ such that $$\Delta \Phi(\rho)\leq 0$$whenever $\rho\geq r_0$. \end{lemma}
\begin{proof}Notice that$$
\nabla_i \rho=\nabla_i (2\sqrt{f})=\frac{\nabla_if}{\sqrt{f}},$$
and\begin{align*}
\Delta \rho&=\nabla^i\nabla_i\rho\\[3pt]
&=\nabla^i\Big(\frac{\nabla_if}{\sqrt{f}}\Big)\\[3pt]
&=\frac{\sqrt{f}\nabla^i\nabla_if-(\nabla_if)(\nabla^i\sqrt{f})}{f}\\[3pt]
&=\frac{\sqrt{f}\Delta f-\frac{|\nabla f|^2}{2\sqrt{f}}}{f}\\[3pt]
&=\frac{\frac{\rho}{2}(\frac{n}{2}-S)-\frac{\frac{\rho^2}{4}-S}{\rho}}{\frac{\rho^2}{4}}\\[3pt]
&=\frac{n-1}{\rho}-\frac{(2\rho^2-4)S}{\rho^3}.
\end{align*}
So
\begin{align*}
\Delta \Phi(\rho)&=\nabla^i\nabla_i\Phi(\rho)\\[3pt]
&=\nabla^i\big(\Phi'( \rho)\nabla_i\rho\big)\\[3pt]
&=\Phi''( \rho)\nabla^i\rho\nabla_i\rho+ \Phi'(\rho)\nabla^i\nabla_i\rho\\[3pt]
&=\Phi''( \rho)|\nabla \rho|^2+\Phi'(\rho)\Delta\rho\\[3pt]
&=\frac{1-n}{\rho^n}|\nabla \rho|^2+\frac{1}{\rho^{n-1}}\Delta\rho\\[3pt]
&=\frac{1-n}{\rho^n}\frac{|\nabla f|^2}{f}+\frac{1}{\rho^{n-1}}\Big(\frac{n-1}{\rho}-\frac{(2\rho^2-4)S}{\rho^3}\Big)\\[3pt]
&=\frac{1-n}{\rho^n}\frac{f-S}{f}+\frac{1}{\rho^{n-1}}\Big(\frac{n-1}{\rho}-\frac{(2\rho^2-4)S}{\rho^3}\Big)\\[3pt]
&=\frac{1-n}{\rho^n}\frac{\frac{\rho^2}{4}-S}{\frac{\rho^2}{4}}+\frac{1}{\rho^{n-1}}\Big(\frac{n-1}{\rho}-\frac{(2\rho^2-4)S}{\rho^3}\Big)\\[3pt]
&=\frac{S(4n-\rho^2)}{\rho^{n+2}}.
\end{align*}Since $S\geq 0$, $\Delta \Phi(\rho)\leq 0$ whenever $\rho\geq 2\sqrt{n}$.\end{proof}

\begin{theorem}\label{three-circles2}Let $u$ be a subharmonic function on $M$. There exists $r_0>0$ such that for any $r_0\leq r_1<r_2<r_3$, we have $$\mathrm{M}_u(r_2)\leq\frac{\mathrm{M}_u(r_1)\big(\Phi(r_3)-\Phi(r_2)\big)+\mathrm{M}_u(r_3)\big(\Phi(r_2)-\Phi(r_1)\big)}{\Phi(r_3)-\Phi(r_1)}.$$ \end{theorem}
\begin{proof}Let $r_0=2\sqrt{n}$. For any $r_0\leq r_1<r_3$, we denote $$A_\rho(r_1,r_3)=\{q\in M|r_1\leq\rho (q)\leq r_3\}.$$Define a function $v$ on $A_\rho(r_1,r_3)$ by $$v=u-a\Phi(\rho)-b,$$ where $$a=\frac{\mathrm{M}_u(r_1)-\mathrm{M}_u(r_3)}{\Phi(r_1)-\Phi(r_3)}>0,$$ and $$b=\mathrm{M}_u(r_1)-\Phi(r_1)\frac{\mathrm{M}_u(r_1)-\mathrm{M}_u(r_3)}{\Phi(r_1)-\Phi(r_3)}.$$Then by Lemma \ref{lemma 3},$$\Delta v=\Delta u-a\Delta\Phi(\rho)\geq 0.$$ Moreover, we have $$v|_{\rho^{-1}(r_1)}\leq 0$$ and $$v|_{\rho^{-1}(r_3)}\leq 0$$Applying Proposition \ref{strong maximum principle} to $v$, we get the desired result.\end{proof}

\begin{lemma}\label{Lemma 4}Let $\Psi(r)=\log\log r$. If $S\geq \frac{n-2}{2}$, then there exists a constant $C_n$ depending only on $n$ such that $\Delta\Psi(\rho)\leq 0$ whenever $\rho \geq C_n$.  \end{lemma}
\begin{proof}Note that\begin{align*}
\Delta\Psi(\rho)&=\Psi''(\rho)\frac{\frac{\rho^2}{4}-S}{\frac{\rho^2}{4}}+\Psi'(\rho)\Big(\frac{n-1}{\rho}-\frac{(2\rho^2-4)S}{\rho^3}\Big)\\[3pt]
&=-\frac{(\log \rho+1)(\rho^2-4S)}{\rho^2(\log \rho)^2}+\frac{1}{\rho\log \rho}\Big(\frac{n-1}{\rho}-\frac{(2\rho^2-4)S}{\rho^3}\Big)\\[3pt]
&=\frac{1}{\rho^4(\log \rho)^2}\big((n-2-2S)\rho^2\log \rho-\rho^2+8S\log \rho+4S\big).
\end{align*}For any $q\in M$, if $\frac{n-2}{2}\leq S(q)\leq \frac{n}{2}$, then $n-2-2S(q)\leq 0$ and $$8S(q)\log \rho(q)+4S(q)\leq \frac{8n}{2}\log \rho(q)+\frac{4n}{2}\leq 4n\rho(q)+2n.$$Hence, when $\rho(q)>5n$, we have $$\Delta\Psi(\rho)(q)\leq \frac{1}{\rho^4(q)(\log \rho(q))^2}(-\rho^2(q)+4n\rho(q)+2n)\leq 0.$$

If $ S(q)>\frac{n}{2}$, then $$(n-2-2S(q))\rho^2(q)\log \rho(q)<-2\rho^2(q)\log \rho(q).$$Note that we always have $S<\frac{\rho^2}{4}$. So\begin{align*}
\lefteqn{(n-2-2S(q))\rho^2(q)\log \rho(q)-\rho^2(q)+8S(q)\log \rho(q)+4S(q)}\\[3pt]
&\leq -2\rho^2(q)\log \rho(q)-\rho^2(q)+2\rho^2(q)\log \rho(q)+\rho^2(q)\\[3pt]
&=0.
\end{align*}This completes our proof.\end{proof}

\begin{example}For $n\geq 3$, let $\mathbb{S}^{n-1}$ be the unit $(n-1)$-sphere, and $g_{can}$ the canonical round metric on $\mathbb{S}^{n-1}$. We endow the cylinder $\mathbb{S}^{n-1}\times \mathbb{R}$ with the product metric $$g(x,t)=dt^2+2(n-2)g_{can}(x),$$where $x\in \mathbb{S}^{n-1}$ and $t\in \mathbb{R}$. Then $(\mathbb{S}^{n-1}\times \mathbb{R},g)$ is a gradient shrinking Ricci soliton with the potential function $$f(x,t)=\frac{t^2}{4}+\frac{n}{2}.$$ Its scalar curvature $S=\frac{n-1}{2}$ satisfies the condition of the above lemma.\end{example}

\begin{theorem}\label{three-circles3}Suppose that $S\geq \frac{n-2}{2}$. Then for any subharmonic function $u$ on $M$ and any $5n<r_1<r_2<r_3$, we have \begin{equation}\label{Soliton}\mathrm{M}_u(r_2)\leq\frac{\mathrm{M}_u(r_1)\big(\log \log r_3- \log \log r_2 \big)+\mathrm{M}_u(r_3)\big(\log \log r_2 - \log \log r_1 \big)}{ \log \log  r_3 -\log \log  r_1}.\end{equation} \end{theorem}
\begin{proof}As above, for any $5n<r_1<r_3$, we denote $$A_\rho(r_1,r_3)=\{q\in M|r_1\leq\rho (q)\leq r_3\}.$$Define a function $w$ on $A_\rho(r_1,r_3)$ by $$w=u-a\Psi(\rho)-b,$$ where $$a=\frac{\mathrm{M}_u(r_1)-\mathrm{M}_u(r_3)}{\Psi(r_1)-\Psi(r_3)}>0,$$ and $$b=\mathrm{M}_u(r_1)-\Psi(r_1)\frac{\mathrm{M}_u(r_1)-\mathrm{M}_u(r_3)}{\Psi(r_1)-\Psi(r_3)}.$$ It follows from Lemma \ref{Lemma 4} that,$$\Delta w=\Delta u-a\Delta\Psi(\rho)\geq 0.$$ Moreover, we have $$w|_{\rho^{-1}(r_1)}\leq 0$$ and $$w|_{\rho^{-1}(r_3)}\leq 0$$Applying Proposition \ref{strong maximum principle} to $w$, we get the desired result.\end{proof}

\begin{remark}Indeed, in the above theorem, we just need $S\geq \frac{n-2}{2}$ holds outside a compact set of $M$. \end{remark}

\begin{corollary}\label{Liouville 1}Keep the assumptions above. Then any bounded subharmonic function on $M$ must be a constant. \end{corollary}
\begin{proof}Assume that $u$ is a bounded subharmonic function on $M$. Fixing $5n<r_1<r_2$ and letting $r_3\rightarrow \infty$ in (\ref{Soliton}), we have immediately that $$\mathrm{M}_u(r_2)\leq \mathrm{M}_u(r_1).$$It means that $u$ attains its maximum at some interior point of $\overline{B}_p(r_2)$. So $u$ is a constant.\end{proof}

\begin{remark}In a recent preprint \cite{MO}, Mai and Ou show that any bounded harmonic function is constant on gradient shrinking Ricci solitons with constant scalar curvature.\end{remark}

\section{K\"{a}hler-Ricci solitons}

In this section, we always assume that $(M,\omega,f)$ be an $n$-dimensional gradient shrinking K\"{a}hler-Ricci soliton with potential function $f$. Then with respect to a unitary frame, we have
\begin{equation}\label{Kahler-Ricci 1}R_{i\overline{j}}+f_{i\overline{j}}=\frac{1}{2}\delta_{i\overline{j}}\end{equation} and \begin{equation}\label{Kahler-Ricci 2}f_{ij}=0.\end{equation}

As before, we set $\rho=2\sqrt{f}$ and $D_r=\{q\in M|\rho(q)\leq r\}$ for $r>0$. Given a holomorphic function $h$ on $M$, we set $$\mathrm{M}_h(r)=\max_{D_r}|h|.$$ By the maximum modulus theorem for holomorphic functions, $\mathrm{M}_h(r)$ is strictly increasing in $r$ unless $h$ is constant.

For any smooth section $X$ of the real tangent bundle $TM$ of $M$, we view it as a smooth section of the complex bundle $TM^\mathbb{C}=TM\otimes \mathbb{C}$ by the canonical embedding $TM\hookrightarrow TM^\mathbb{C}$. We denote by $X^{1,0}=X-\sqrt{-1}JX$ the projection of $X$ to $T^{1,0}M$. Here $J$ is the complex structure on $M$ and we omit the factor $\frac{1}{2}$ for simplicity.

\begin{lemma}\label{lemma 5}For any $\varphi\in C^2(M,\mathbb{R})$ and any smooth real vector fields $X$ and $Y$ on $M$, we have $$\mathrm{Hess}(\varphi)(X^{1,0},\overline{Y^{1,0}})=\partial\overline{\partial}\varphi(X^{1,0},\overline{Y^{1,0}}).$$ \end{lemma}
\begin{proof}At each point $q\in M$, we choose a holomorphic coordinate system $\{z^i\}$ at $p$. Locally, we write $X^{1,0}=X^i\frac{\partial}{\partial z^i}$ and $Y^{1,0}=Y^i\frac{\partial}{\partial z^i}$. Since $M$ is a K\"{a}hler manifold, we have $\nabla_{\frac{\partial}{\partial z^i}}\frac{\partial}{\partial \overline{z}^j}=0$. Thus\begin{align*}
\mathrm{Hess}(\varphi)(X^{1,0},\overline{Y^{1,0}})&=X^{1,0}\overline{Y^{1,0}}\varphi-\big(\nabla_{X^{1,0}}\overline{Y^{1,0}}\big)\varphi\\[3pt]
&=X^i\frac{\partial}{\partial z^i}\Big(\overline{Y^j}\frac{\partial \varphi}{\partial \overline{z}^j}\Big)-X^i\Big(\frac{\partial \overline{Y^j}}{\partial z^i} \frac{\partial \varphi}{\partial \overline{z}^j}+\overline{Y^j}\Big(\nabla_{\frac{\partial}{\partial z^i}}\frac{\partial}{\partial \overline{z}^j}\Big)\varphi\Big)\\[3pt]
&=X^i\overline{Y^j}\frac{\partial^2 \varphi}{ \partial z^i\partial \overline{z}^j}\\[3pt]
&=\frac{\partial^2 \varphi}{ \partial z^i\partial \overline{z}^j}dz^i\wedge d\overline{z}^j(X^{1,0},\overline{Y^{1,0}})\\[3pt]
&=\partial\overline{\partial}\varphi(X^{1,0},\overline{Y^{1,0}}).
\end{align*}\end{proof}

\begin{lemma}\label{lemma 6}Set $e_\rho=\frac{\nabla \rho}{|\nabla \rho|}$. If $\mathrm{Ric}(e_\rho^{1,0}, \overline{e_\rho^{1,0}})\geq \frac{S}{f}$, then $$\mathrm{Hess}(\log\rho)(\nabla \rho^{1,0},\overline{\nabla \rho^{1,0}})\leq 0.$$\end{lemma}
\begin{proof}Since $\rho=2\sqrt{f}$, it is easy to compute that for any smooth real vector fields $X$ and $Y$ on $M$, $$\mathrm{Hess}(\log\rho)(X,Y)=\frac{1}{\rho}\mathrm{Hess}(\rho)(X,Y)-\frac{(X\rho)(Y\rho)}{\rho^2},$$ and $$\mathrm{Hess}(\rho)(X,Y)=\frac{1}{\sqrt{f}}\mathrm{Hess}(f)(X,Y)-\frac{(Xf)(Yf)}{2\sqrt{f^3}}.$$

For a smooth real-valued function $\varphi$ on $M$, we have $$\overline{\mathrm{Hess}(\varphi)(\nabla \rho^{1,0},\overline{\nabla \rho^{1,0}})}=\mathrm{Hess}(\varphi)(\overline{\nabla \rho^{1,0}},\nabla \rho^{1,0})=\mathrm{Hess}(\varphi)(\nabla \rho^{1,0},\overline{\nabla \rho^{1,0}}),$$ which asserts that $\mathrm{Hess}(\varphi)(\nabla \rho^{1,0},\overline{\nabla \rho^{1,0}})$ is real. So, by the $\mathbb{C}$-bilinearity of hessians, we have\begin{align*}
\mathrm{Hess}(\log\rho)(\nabla \rho^{1,0},\overline{\nabla \rho^{1,0}})&=\mathrm{Hess}(\log\rho)(\nabla \rho,\nabla \rho)+\mathrm{Hess}(\log\rho)(J\nabla \rho,J\nabla \rho)\\[3pt]
&=\frac{1}{\rho}\mathrm{Hess}(\rho)(\nabla \rho,\nabla \rho)-\frac{\big((\nabla \rho )\rho\big)^2}{\rho^2}\\[3pt]
&\ \ \ +\frac{1}{\rho}\mathrm{Hess}(\rho)(J\nabla \rho,J\nabla \rho)-\frac{\big((J\nabla \rho )\rho\big)^2}{\rho^2}\\[3pt]
&=\frac{1}{2f}\mathrm{Hess}(f)(\nabla \rho,\nabla \rho)-\frac{\big((\nabla \rho) f\big)^2}{4f^2}-\frac{|\nabla \rho |^4}{\rho^2}\\[3pt]
&\ \ \ +\frac{1}{2f}\mathrm{Hess}(f)(J\nabla \rho,J\nabla \rho)-\frac{\big((J\nabla \rho) f\big)^2}{4f^2}-\frac{\langle J\nabla \rho,\nabla\rho\rangle^2}{\rho^2}\\[3pt]
&=\frac{1}{2f}\mathrm{Hess}(f)(\nabla \rho,\nabla \rho)-\frac{\langle \nabla \rho,\nabla f\rangle^2}{4f^2}-\frac{|\nabla \rho |^4}{\rho^2}\\[3pt]
&\ \ \ +\frac{1}{2f}\mathrm{Hess}(f)(J\nabla \rho,J\nabla \rho)\\[3pt]
&=\frac{1}{2f}\mathrm{Hess}(f)(\nabla \rho,\nabla \rho)-\frac{|\nabla f |^4}{2f^3}+\frac{1}{2f}\mathrm{Hess}(f)(J\nabla \rho,J\nabla \rho)\\[3pt]
&=\frac{1}{2f}\Big(\frac{|\nabla \rho|^2}{2}-\mathrm{Ric}(\nabla \rho)+\frac{|J\nabla \rho|^2}{2}-\mathrm{Ric}(J\nabla \rho)\Big)-\frac{|\nabla f |^4}{2f^3}\\[3pt]
&=\frac{|\nabla \rho|^2}{2f}-\frac{|\nabla f |^4}{2f^3}-\frac{1}{2f}\big(\mathrm{Ric}(\nabla \rho)+\mathrm{Ric}(J\nabla \rho)\big)\\[3pt]
&=\frac{|\nabla f|^2(f-|\nabla f |^2)}{2f^3}-\frac{1}{2f}\big(|\nabla \rho|^2\mathrm{Ric}(e_\rho)+|J\nabla \rho|^2\mathrm{Ric}(Je_\rho)\big)\\[3pt]
&= \frac{|\nabla f|^2S}{2f^3}-\frac{|\nabla \rho|^2}{2f^2}(\mathrm{Ric}(e_\rho)+\mathrm{Ric}(Je_\rho))\\[3pt]
&= \frac{|\nabla f|^2}{2f^3}\big(S-f(\mathrm{Ric}(e_\rho)+\mathrm{Ric}(Je_\rho))\big).
\end{align*}In the above calculation, we used the fact $\langle JX,JX\rangle=\langle X,X\rangle$ and $\langle JX,X\rangle=0$ for any real smooth vector field $X$ on $M$. Thus, if $\mathrm{Ric}(e_\rho^{1,0}, \overline{e_\rho^{1,0}})=\mathrm{Ric}(e_\rho)+\mathrm{Ric}(Je_\rho)\geq \frac{S}{f}$, we have $\mathrm{Hess}(\log\rho)(\nabla \rho^{1,0},\overline{\nabla \rho^{1,0}})\leq 0$.\end{proof}

\begin{remark}(i) The curvature pinched condition $\mathrm{Ric}(e_\rho^{1,0}, \overline{e_\rho^{1,0}})\geq \frac{S}{f}$ seems to be quite strong. However, it can hold in some certain kinds of shrinking solitons. For example, the second author in \cite{Z} can show a stronger result.

(ii) Consider the gradient shrinking K\"{a}hler-Ricci soliton $(\mathbb{C}^n,\overline{g},f)$, where $\overline{g}$ is the canonical Eucldean metric and $f(z)=\frac{|z|^2}{4}$ for $z\in \mathbb{C}^n$. Since $(\mathbb{C}^n,\overline{g})$ is flat, it satisfies the condition in the above lemma.\end{remark}

\begin{theorem}Suppose that $\nabla f\neq 0$ and $\mathrm{Ric}(e_\rho^{1,0}, \overline{e_\rho^{1,0}})\geq \frac{S}{f}$ on $M\setminus D_{r_0}$ for some $r_0>0$. Then, for any holomorphic function $h$ on $M$ and any $r_0\leq r_1<r_2<r_3$, we have \begin{equation}\label{holomorphic}\log\mathrm{M}_h(r_2)\leq\frac{\log\mathrm{M}_h(r_1)(\log r_3-\log r_2)+\log\mathrm{M}_h(r_3)(\log r_2-\log r_1)}{\log r_3-\log r_1}.\end{equation}\end{theorem}
\begin{proof}Set $$G=(\log r_3-\log r_1)\log|h|$$ and $$F=\log\frac{r_3}{\rho} \log\mathrm{M}_h(r_1)+\log\frac{\rho }{r_1}\log\mathrm{M}_h(r_3).$$We claim that $G-F\leq 0$ on $A_\rho(r_1,r_3)=\{q\in M|r_1\leq \rho (q)\leq r_3\}$, which implies our result.

For any $0<\varepsilon<r_1$, we define $$F_\varepsilon=a_\varepsilon \log(\rho-\varepsilon)+b_\varepsilon,$$where$$a_\varepsilon=\frac{\big(\log r_3 -\log r_1)(\mathrm{M}_h(r_1)-\mathrm{M}_h(r_3))}{\log(r_1-\varepsilon)-\log(r_3-\varepsilon)}>0,$$ and $$b_\varepsilon=\big(\log r_3 -\log r_1 \big)\frac{\mathrm{M}_h(r_1)\log(r_3-\varepsilon)-\mathrm{M}_h(r_3)\log(r_1-\varepsilon)}{\log(r_3-\varepsilon)-\log(r_1-\varepsilon)}.$$ Then $F_\varepsilon(q)=(\log r_3-\log r_1)\mathrm{M}_h(r_1)$ for $q\in \rho^{-1}(r_1)$ and $F_\varepsilon(q)=(\log r_3 -\log r_1)\mathrm{M}_h(r_3)$ for $q\in \rho^{-1}(r_3)$. Obviously, $F_\varepsilon\rightarrow F$ as $\varepsilon\rightarrow 0^+$.

We are going to show by contradiction that $G-F_\varepsilon\leq 0$ on $A_\rho(r_1,r_3)$. Assume that $(G-F_\varepsilon)(q)>0$ for some $q\in A_\rho(r_1,r_3)$. Since $G-F_\epsilon$ is continuous on $A_\rho(r_1,r_3)$, it attains its maximum at some $\tilde{q}_\varepsilon\in A_\rho(r_1,r_3)$. Moreover $(G-F_\varepsilon)(\tilde{q}_\varepsilon)>0$. Since $G-F_\varepsilon$ is non-positive on the boundary of $A_\rho(r_1,r_3)$. $\tilde{q}_\varepsilon$ is an interior point of $A_\rho(r_1,r_3)$.

From calculus, the Hessian $\mathrm{Hess}(G-F_\varepsilon)$ of $G-F_\varepsilon$ at $\tilde{q}_\varepsilon$ is non positive definite. In particular, we have $$\mathrm{Hess}(G-F_\varepsilon)_{\tilde{q}_\varepsilon}(\nabla \rho^{1,0}|_{\tilde{q}_\varepsilon},\overline{\nabla \rho^{1,0}}|_{\tilde{q}_\varepsilon})\leq 0.$$ Observe that $|h(\tilde{q}_\varepsilon)|\neq 0$, otherwise $G(\tilde{q}_\varepsilon)=-\infty$, a contradiction. It follows from Lemma \ref{lemma 5} and the Poincar\'{e}-Lelong equation $\frac{\sqrt{-1}}{2\pi}\partial\overline{\partial}\log |h|^2=[D]$, we have$$\mathrm{Hess}(G)_{\tilde{q}_\varepsilon}(\nabla \rho^{1,0}|_{\tilde{q}_\varepsilon},\overline{\nabla \rho^{1,0}}|_{\tilde{q}_\varepsilon})=\partial\overline{\partial} G_{\tilde{q}_\varepsilon}(\nabla \rho^{1,0}|_{\tilde{q}_\varepsilon},\overline{\nabla \rho^{1,0}}|_{\tilde{q}_\varepsilon})=0.$$Here $[D]$ is the divisor of $h$. Therefore $\mathrm{Hess}(F_\varepsilon)_{\tilde{q}_\varepsilon}(\nabla \rho^{1,0}|_{\tilde{q}_\varepsilon},\overline{\nabla \rho^{1,0}}|_{\tilde{q}_\varepsilon})\geq 0$.

On the other hand, by a similar calculation as in Lemma \ref{lemma 6}, we have\begin{align*}
\lefteqn{\mathrm{Hess}(F_\varepsilon)(\nabla \rho^{1,0},\overline{\nabla \rho^{1,0}})}\\[3pt]
&=a_\varepsilon\mathrm{Hess}\big(\log(\rho-\varepsilon)\big)(\nabla \rho^{1,0},\overline{\nabla \rho^{1,0}})\\[3pt]
&=\frac{a_\varepsilon |\nabla f|^2}{(\rho-\varepsilon)f^{5/2}}\Big(f-\frac{|\nabla f|^2(\rho-\varepsilon+2\sqrt{f})}{2(\rho-\varepsilon)}-f\mathrm{Ric}(e_\rho^{1,0}, \overline{e_\rho^{1,0}})\Big)\\[3pt]
&<\frac{a_\varepsilon |\nabla f|^2}{(\rho-\varepsilon)f^{5/2}}\big(f-|\nabla f|^2-f\mathrm{Ric}(e_\rho^{1,0}, \overline{e_\rho^{1,0}})\big),
\end{align*}where the last inequality follows from the assumption $\nabla f\neq 0$ and the fact $\frac{\sqrt{f}}{\rho-\varepsilon}>\frac{1}{2}$. By the condition $\mathrm{Ric}(e_\rho^{1,0}, \overline{e_\rho^{1,0}})\geq \frac{S}{f}$, we have $\mathrm{Hess}(F_\varepsilon)_{\tilde{q}_\varepsilon}(\nabla \rho^{1,0}|_{\tilde{q}_\varepsilon},\overline{\nabla \rho^{1,0}}|_{\tilde{q}_\varepsilon})< 0$, a contradiction.

By letting $\varepsilon\rightarrow 0^+$, we have $G-F\leq 0$ on $A(r_1,r_2)$, as desired.
\end{proof}

By a similar argument as before, we have immediately that

\begin{corollary}\label{Liouville 2}Keep the assumptions in the above theorem. Then any bounded holomorphic function on $M$ must be a constant. \end{corollary}

\bibliographystyle{amsplain}

\begin{thebibliography}{10}
\bibitem{C} Cao, H.-D.: \textit{Recent progress on Ricci
solitons}. In: Recent advances in geometric analysis, Adv. Lect. Math. (ALM), \textbf{11}, pp. 1--38. Int. Press, Somerville, MA, 2010

\bibitem{CZ}Cao, H.-D. and Zhou, D.: \textit{On complete gradient shrinking Ricci solitons}. J. Differential Geom. \textbf{85},
175--186 (2010)

\bibitem{CHL}Chang, S.-C., Han, Y., Lin, C.: \textit{On the three-circle theorem and its applications in Sasakian manifolds}. Calc. Var. 58, 101 (2019)

\bibitem{CT}Cheeger, J., Tian, G.: \textit{On the cone structure at infinity of Ricci flat manifolds with Euclidean volume
growth and quadratic curvature decay}. Invent. Math. \textbf{118}(1), 493--571 (1994)

\bibitem{C} Chen, B.-L.: \textit{Strong uniqueness of the Ricci flow}. J. Differential Geom. \textbf{82}, 363--382 (2009)

\bibitem{CM}Colding, T., Minicozzi II ,W.: \textit{Harmonic functions with polynomial growth}. J. Differential Geom. \textbf{45}, 1-77 (1997)

\bibitem{EMT} Enders, J., Muller, R., Topping, P.: \textit{On type-I singularities in Ricci flow}. Comm. Anal. Geom. \textbf{19}, no. 5, 905--922 (2011)

\bibitem{H}Hamilton, R.: \textit{The formation of singularities in the Ricci flow}. In: Surveys in differential geometry II. pp. 7--136. International Press Boston (1993)

\bibitem{L}Liu, G.: \textit{Three-circle theorem and dimension estimate for holomorphic functions on
K\"{a}hler manifolds}. Duke Math. J. \textbf{165}, no. 15, 2899--2919 (2016)


\bibitem{MO}Mai W, Ou J. \textit{Liouville Theorem on Ricci shrinkers with constant scalar curvature and its application}. arXiv: 2208.07101

\bibitem{O}Ou, J: \textit{Three circle theorems for eigenfunctions on complete shrinking gradient Ricci solitons with constant scalar curvature}. J. Math. Anal. Appl. \textbf{464}, 1243--1259 (2018)

\bibitem{P}Petersen, P.: \textit{Riemannian Geometry}. Graduate Texts in Mathematics, vol. \textbf{171}, 3rd edn.
Springer, New York (2016)

\bibitem{PW}Protter, P. W., Weinberger, H. F.: \textit{Maximum principle in differential equations}. Springer-Verlag New York (1984)

\bibitem{S}Simon, L.: \textit{Asymptotics for a class of nonlinear evolution equations, with applications to geometric
problems}. Ann. Math. \textbf{118}(2), 3, 525--571 (1983)

\bibitem{MW}Munteanu, O., Wang, J.: \textit{Holomorphic functions on K\"{a}hler-Ricci solitons}. J. London Math. Soc. \textbf{89}, 817--831 (2014)

\bibitem{X}Xu, G: \textit{Three circles theorems for harmonic functions}. Math. Ann. \textbf{366}, 1281--1317 (2016)

\bibitem{YZ}Yu, C., Zhang, C.: \textit{Three circle theorem on almost Hermitian manifolds and applications}. Calc. Var. \textbf{61}, 184 (2022)


\bibitem{Z} Zhang, Z.-H.: \textit{A gap theorem of four-dimensional gradient shrinking solitons}. Comm. Anal. Geom. \textbf{28}, no 3, 729--742 (2020)
\end{thebibliography}

\end{document}